\documentclass{amsart}
\usepackage{amsfonts}
\usepackage{amsmath}
\usepackage{setspace}
\usepackage{amsthm}
\usepackage{amssymb}
\usepackage[all]{xy}
\usepackage{color}
\usepackage{soul}
\usepackage{mathrsfs}
\usepackage{colortbl}
\usepackage{rotating}
\usepackage{tikz}
\usetikzlibrary{arrows.meta, decorations.pathmorphing, backgrounds, positioning, fit, petri, patterns}

\definecolor{linkylink}{RGB}{200,0,0}
\usepackage[colorlinks,citecolor=blue,linkcolor=linkylink]{hyperref}

\newcommand{\Z}{{\mathbb{Z}}}
\newcommand{\E}{\mathcal E}
\newcommand{\Zcal}{\mathcal Z}
\newcommand{\R}{{\mathbb{R}}}

\newcommand{\C}{\mathcal C}
\newcommand{\D}{\mathcal D}

\newcommand{\pushout}{\ar@{}[ul(0.35)]|-{\ulcorner}}
\newcommand{\pullback}{\ar@{}[dr(0.35)]|-{\lrcorner}}

\newcommand{\add}{\mathop{\text{add}}}

\DeclareMathOperator{\rep}{rep}
\newcommand{\AR}{A_{\R}}
\newcommand{\ARS}{A_{\R{,}S}}
\newcommand{\kk}{\ensuremath{\Bbbk}}
\newcommand{\repAR}{\rep_{\kk}(\AR)}

\newcommand{\cfrak}{\mathfrak{c}}
\newcommand{\A}{\mathcal A}
\DeclareMathOperator{\adic}{\mathsf{adic}}
\DeclareMathOperator{\Prufer}{\mathsf{Pr\ddot{u}fer}}

\newcommand{\inftyclosed}{\overline{\infty}}

\newcommand{\DbAR}{{\mathcal{D}^b(\AR)}}

\newcommand{\CAR}{{\mathcal{C}(\AR)}}
\newcommand{\CARS}{{\mathcal{C}(\ARS)}}

\newcommand{\Cinftyclosed}{\C(A_{\inftyclosed})}
\newcommand{\Ninftyclosed}{\NN_{\inftyclosed}}

\newcommand{\Finftyinftyclosed}{F_\infty^{\inftyclosed}}
\newcommand{\etainftyinftyclosed}{{\eta_\infty^{\inftyclosed}}}

\newcommand{\KsplitCARS}{{K_0^{\text{split}}(\CARS)}}

\newcommand{\theory}[2]{\mathscr{T}_{#1}(#2)}
\newcommand{\struc}[2]{\mathscr{S}_{#1}(#2)}
\newcommand{\PP}{\mathbf P}
\newcommand{\QQ}{\mathbf Q}
\newcommand{\NN}{\mathbf N}
\newcommand{\EE}{\mathbf E}
\newcommand{\TT}{\mathbf T}
\newcommand{\iinftyclosed}{\overline{i,\infty}}
\newcommand{\jinftyclosed}{\overline{j,\infty}}

\DeclareMathOperator{\Ind}{Ind}
\DeclareMathOperator{\Hom}{Hom}

\newtheorem{lemma}{Lemma}[section]

\newtheorem{thm}{Theorem}

\theoremstyle{definition}
\newtheorem{definition}[lemma]{Definition}
\newtheorem{remark}[lemma]{Remark}
\newtheorem{example}[lemma]{Example}

\makeatletter
\@namedef{subjclassname@2020}{\textup{2020} Mathematics Subject Classification}
\makeatother

\title[Cluster Theories and Cluster Structures of Type A]{Cluster Theories and Cluster Structures\\of Type A}
\subjclass[2020]{05E10, 13F60, 16G20}
\author{Job Daisie Rock}
\address{Ghent University, Ghent, Belgium}
\email{\texttt{jobdrock@gmail.com}}
\date{\today}

\begin{document}
\maketitle

\begin{abstract}
In the present paper we examine the relationship between several type $A$ cluster theories and structures.
We define a 2D geometric model of a cluster theory, which generalizes cluster algebras from surfaces, and encode several existing type $A$ cluster theories into a 2D geometric model.
We review two other cluster theories of type $A$.
Then we introduce an abstraction of cluster structures.
We prove two results: the first relates several existing type $A$ cluster theories and the second relates some of these cluster structures using the new abstraction.
\end{abstract}


\section{Introduction}
\subsection*{History}
	Cluster algebras were introduced by Fomin and Zelevinsky in 2002 \cite{FZ02}.
	Since then, there have been efforts to study cluster structures of cluster algebras via combinatorial methods on surfaces.
	Caldero, Chapaton, and Schiffler proved that type $A_n$ cluster structures are encoded in a cluster category of triangulations of the $(n+3)$-gon \cite{CCS06}.
	Their cluster category is equivalent to the type introduced by Buan, Marsh, Reineke, et al in \cite{BMRRT06} and later generalized by Buan, Iyama, Reiten, and Scott in \cite{BIRS09}.
	Fomin, Shapiro, and Thurston studied cluster structures using triangulated surfaces in general \cite{FST08}.
	For the state of the art at the time of writing, we refer the reader to \cite[Chapter 4]{A21}.
	
	Cluster structures from triangulations of polygons were generalized by Holm and J{\o}rgensen to triangulations of the infinity-gon, whose vertices are indexed by $\Z$ \cite{HJ12}.
	Igusa and Todorov generalized that cluster structure to the $i$-infinity-gon (in the terminology of this paper), whose vertices are indexed by $i\geq 1$ many copies of $\Z$ \cite{IT15a}.
	Baur and Graz introduced a cluster structure using triangulations of the completed infinity-gon, whose vertices are $\Z\cup\{\pm\infty\}$ \cite{BG18}.
	Paquette and Y{\i}ld{\i}r{\i}m generalized Igusa and Todorov's model by completing the $i$-infinity gon with an additional $i$ vertices \cite{PY21}.
	Baur and Graz also did this for the 1-infinity-gon \cite{BG18}.
	Igusa and Todorov also generalized triangulations of the polygon to ideal triangulations of the hyperbolic plane \cite{IT15b}.
	One may think of this as triangulations of the ``$\R$-gon.''
	
	Igusa, Todorov, and the author generalized cluster structures to cluster theories \cite{IRT22b}.
	Any existing cluster structure is a cluster theory; they also introduced $\EE$-cluster theories for continuous quivers of type $A$.
	The author then created a geometric models of $\EE$-cluster theories in \cite{R20}.
	
	Chapoton, Fomin, and Zelevinsky proved that the associahedron in $n$ dimensions encodes the cluster structure of type $A_n$ \cite{CFZ02}.
	Kulkarni, Matherne, Mousavand, and the author defined a $\TT$-cluster theory to obtain a continuous generalization of an associahedron, which allows embeddings from all the finite type $A_n$ associahedra \cite{KMMR21}.
	
\subsection*{Contributions}
	In Section~\ref{sec:cluster theories (big)} we review the definition of a cluster theory (Definition~\ref{def:cluster theory}), which generalizes the combinatorics of a cluster structure, and the definition of an embedding of cluster theories (Definition~\ref{def:cluster embedding}).
	We then define a generalization of triangulations of surfaces (Definition~\ref{def:2D geometric model of a cluster theory}) to capture the cluster combinatorics and encode several existing type $A$ cluster theories in this way (Section~\ref{sec:geometric models for embeddings}).
	Then we recall two more existing type $A$ cluster theories (Section~\ref{sec:other theories}).
	We finish Section~\ref{sec:cluster theories (big)} by introducing an abstract notion of a cluster structure and an embedding of such cluster structures (Definitions~\ref{def:cluster structure} and \ref{def:restrict to cluster structure}).

	In Sections \ref{sec:embeddings of cluster theories} and \ref{sec:strutures in theories} we prove Theorems~\ref{thm:intro:theories} and \ref{thm:intro:structures}, respectively.
	\begin{thm}\label{thm:intro:theories}
	For any $2\leq m <n$ and $2\leq i < j$, the diagram on page \pageref{fig:intro diagram} is a commutative diagram of embeddings of cluster theories.
	\begin{sidewaysfigure*}
	\begin{align*}
		{~} \\ {~} \\ {~} \\ {~} \\ {~} \\ {~} \\ {~} \\
		{~} \\ {~} \\ {~} \\ {~} \\ {~} \\ {~} \\ {~} \\
		{~} \\ {~} \\ {~} \\ {~} \\ {~} \\ {~} \\ {~} \\
		{~} \\
	\end{align*}
	\begin{displaymath}
	\xymatrix@C=8ex{
		\theory{\NN_m}{\C(A_m)} \ar[r] \ar[d] &
		\theory{\NN_\infty}{\C(A_\infty)}\cong \theory{\NN_{1,\infty}}{\C(A_{1,\infty})} \ar[r] \ar[d] &
		\theory{\NN_{i,\infty}}{\C(A_{i,\infty})} \ar[r] \ar[d] &
		\theory{\NN_{j,\infty}}{\C(A_{j,\infty})} \ar[r] \ar[d] &
		\theory{\NN_\pi}{\C_\pi} \ar[d]
		\\
		\theory{\NN_n}{\C(A_n)} \ar[r] \ar[ur] &
		\theory{\NN_{\overline{1,\infty}}}{\C(A_{\overline{1,\infty}})} \ar[r] \ar[d] &
		\theory{\NN_{\iinftyclosed}}{\C(A_{\iinftyclosed})} \ar[r] &
		\theory{\NN_{\jinftyclosed}}{\C(A_{\jinftyclosed})} \ar[r] \ar[ur] &
		\theory{\EE}{\CAR}
		\\
		&
		\theory{\Ninftyclosed}{\Cinftyclosed} \ar[r] &
		\theory{\NN_{\overline{2,\infty}}}{\C(A_{\overline{2,\infty}})} \ar[u]
		\\
		\theory{\NN_m}{\C(A_m)} \ar[r] \ar@/^4ex/[rr] &
		\theory{\NN_n}{\C(A_n)} \ar[r] &
		\theory{\TT}{\C_\Zcal}
	}
	\end{displaymath}
	\caption{A commutative diagram of embeddings of cluster theories.
		Theories $\theory{\NN_m}{\C(A_m)}$ and $\theory{\NN_n}{\C(A_n)}$ are from \cite{CCS06}.
		Theory $\theory{\NN_\infty}{\C(A_\infty)}$ is from \cite{HJ12}.
		Theories $\theory{\NN_{1,\infty}}{\C(A_{1,\infty})}$, $\theory{\NN_{i,\infty}}{\C(A_{i,\infty})}$, and $\theory{\NN_{j,\infty}}{\C(A_{j,\infty})}$ are from \cite{IT15a}.
		Theory $\theory{\Ninftyclosed}{\Cinftyclosed}$ is from \cite{BG18}.
		Theories $\theory{\NN_{\overline{1,\infty}}}{\C(A_{\overline{1,\infty}})}$, $\theory{\NN_{\overline{2,\infty}}}{\C(A_{\overline{2,\infty}})}$, $\theory{\NN_{\iinftyclosed}}{\C(A_{\iinftyclosed})}$, and $\theory{\NN_{\jinftyclosed}}{\C(A_{\jinftyclosed})}$ are from \cite{PY21}.
		Theory $\theory{\NN_\pi}{\C_\pi}$ is from \cite{IT15b}.
		Theory $\theory{\EE}{\CAR}$ is from \cite{IRT22b}.
		Theory $\theory{\TT}{\C_\Zcal}$ is from \cite{KMMR21}.}\label{fig:intro diagram}
	\end{sidewaysfigure*}
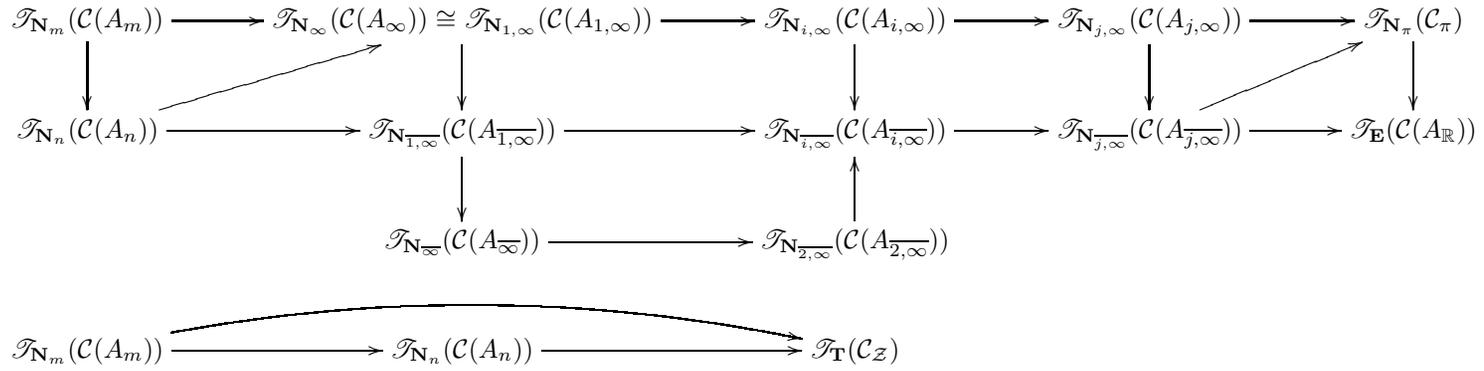
	\end{thm}
	
	In our notation below, we replace $\mathscr T$ with $\mathscr S$ (see the diagram on page \pageref{fig:intro diagram}).
	
	\begin{thm}\label{thm:intro:structures}
		There is a commutative diagram of embeddings of abstract cluster structures:
	\begin{displaymath}
		\xymatrix{
			\struc{\NN_m}{\C(A_m)} \ar[r] \ar[d] &
			\struc{\NN_\infty}{\C(A_\infty)} \ar[r] \ar[dr] \ar[d] &
			\struc{\NN_{1,\infty}}{\C(A_{1,\infty})} \ar[d] \ar[r] &
			\struc{\NN_{j,\infty}}{\C(A_{j,\infty})}
			\\
			\struc{\NN_n}{\C(A_n)} \ar[r] \ar[ur] &
			\struc{\NN_\pi}{\C_\pi} &
			\struc{\NN_{i,\infty}}{\C(A_{i,\infty})}. \ar[ur]
			&
		}
	\end{displaymath}
	\end{thm}

\subsection*{Conventions}
	We have a fixed field $\kk$ throughout.
	A ``Krull--Schmidt category'' is a ``skeletally small Krull--Schmidt additive category.''
	By $\Ind(\C)$ we denote the set of isomorphism classes of indecomposable objects in a Krull--Schmidt category $\C$. 
	
	Let $a<b\in\R\cup\{\pm \infty\}$.
	By $|a,b|$ we denote an interval subset of $\R$ whose endpoints are $a$ and $b$ but whose inclusion of $a$ or $b$ is not known or not relevant.
	
	We say ``arc'' in reference lines between non-consecutive vertices of a polygon where others say ``diagonal.''
	However, we still use ``diagonals'' for quadrilaterals when we talk about mutation.

\subsection*{Acknowledgements}
	Some of this work was completed while the author was a graduate student at Brandeis University and some of this work was completed when the author was at the Hausdorff Research Institute for Mathematics.
	The author thanks both institutions for their hospitality.
	The author was also partly supported by UGent BOF grant BOF/STA/201909/038 and FWO grants G023721N and G0F5921N.

\section{Cluster Theories and Cluster Structures}\label{sec:cluster theories (big)}
We first review the definition of cluster theories in Section~\ref{sec:cluster theories}. 
In Section~\ref{sec:2D geometric models}, we define 2D geometric models of cluster theories and in Section~\ref{sec:geometric models for embeddings} we use those to define cluster theories from several sources \cite{BG18, CCS06, HJ12, IT15a, IT15b, PY21}. 
We review the cluster theories from \cite{IRT22b,KMMR21} in Section \ref{sec:other theories}.
In Section~\ref{sec:abstract cluster structures} we define abstract cluster structures and embeddings of abstract cluster structures.

\subsection{Cluster Theories}\label{sec:cluster theories}
In this section we review cluster theories and embeddings of cluster theories, from \cite{IRT22b}, and we review isomorphisms and weak equivalences of cluster theories, from \cite{R20}.
Throughout this section, let $\C$ and $\D$ be Krull--Schmidt categories.
\begin{definition}\label{def:cluster theory}
	Let $\PP$ be a pairwise compatibility condition $\Ind(\C)$.
	For every maximally $\PP$-compatible set $T\subset\Ind(\C)$, and every $X\in T$, suppose there exists 0 or 1 indecomposables $Y\in\Ind(\C)$ such that $\{X,Y\}$ is not $\PP$-compatible but $(T\setminus\{X\})\cup\{Y\}$ is $\PP$-compatible.
	Then
	\begin{itemize}
	\item We call the maximally $\PP$-compatible sets \ul{$\PP$-clusters}.
	\item We call a function of the form $\mu:T\to (T\setminus\{X\})\cup\{Y\}$ such that $\mu Z=Z$ when $Z\neq X$ and $\mu X=Y$ a \ul{$\PP$-mutation}.
	\item The subcategory $\theory{\PP}{\C}$ of $\mathcal S\text{et}$ whose objects are $\PP$-clusters and whose morphisms are generated by $\PP$-mutations (and identity functions) is called \ul{the $\PP$-cluster theory of $\C$}.
	\item The functor $I_{\PP,\C}:\theory{\PP}{\C}\to \mathcal S\text{et}$ is the inclusion of the subcategory.
	\end{itemize}
\end{definition}

Every cluster structure in the sense of \cite{BMRRT06,BIRS09} yields a cluster theory.

\begin{remark}\label{rmk:cluster theory consequences}
	We note three things immediately from Definition~\ref{def:cluster theory}.
	\begin{itemize}
	\item The set $(T\setminus \{X\}) \cup \{Y\}$ must be maximally $\PP$-compatible.
	\item The category $\theory{\PP}{\C}$ is small.
	\item The pairwise compatibility condition $\PP$ determines the cluster theory.
	\end{itemize}
\end{remark}

In light of Remark \ref{rmk:cluster theory consequences}, we say $\PP$ induces the $\PP$-cluster theory of $\C$.

\begin{definition}\label{def:cluster embedding}
	Let $\theory{\PP}{\C}$ and $\theory{\QQ}{\D}$ be cluster theories.
	Let $F:\theory{\PP}{\C} \to \theory{\QQ}{\D}$ be a functor such that $F$ is an injection on objects and an injection from $\PP$-mutations to $\QQ$-mutations.
	Let $\eta: I_{\PP,\C} \to  I_{\QQ, \D}\circ F$ be a natural transformation such that the morphisms $\eta_T: I_{\PP,\C}(T) \to I_{\QQ, \D}\circ F(T)$ are all injections.
	Then we call $(F,\eta):\theory{\PP}{\C} \to \theory{\QQ}{\D}$ an \ul{embedding of cluster theories}.
\end{definition}

\begin{definition}\label{def:equivalence of cluster theories}
Let $\theory{\PP}{\C}$ and $\theory{\QQ}{\D}$ be cluster theories.
An \ul{isomorphism of cluster theories} is an embedding of cluster theories $(F,\eta):\theory{\PP}{\C} \to \theory{\QQ}{\D}$ such that $F$ is an isomorphism of categories and each $\eta_T$ is an isomorphism.
\end{definition}

\subsection{2D Geometric Models}\label{sec:2D geometric models}
In this section we define 2D geometric models of cluster theories (Definition~ \ref{def:2D geometric model of a cluster theory}).
The ``2D'' indicates that this is a generalization of triangulations of surfaces, a good place to obtain cluster algebras and cluster categories (we refer the reader to \cite[Chapter 4]{A21} for more information).
We focus on the combinatorial aspect of arcs on a surface and their crossings.

\begin{definition}\label{def:geometric model}
Let $\A$ be a set.
Let $\cfrak:\A\times\A\to\{0,1\}$ be a function such that
\begin{itemize}
	\item $\cfrak(\alpha,\beta) = \cfrak(\beta,\alpha)$ for all $\alpha,\beta\in\A$ and
	\item $\cfrak(\alpha,\alpha)=1$ for all $\alpha\in\A$.
\end{itemize}
We call $\A$ the set of \ul{arcs} and $\cfrak$ the \ul{crossing function} and write the pair as $(\A,\cfrak)$.
\end{definition}

\noindent \textbf{!!} Notice we consider each arc to cross itself. This simplifies Definition~\ref{def:additive categorification}.

\begin{example}[Polygon]\label{xmp:polygon geometric model}
	Let $\E$ be the set of natural numbers $\{1,\ldots,n\}$ and $\A_n$ the set of pairs $(i,j)\in\E\times\E$ such that $2\leq j-i \leq n-1$.
	Define $\cfrak_n:\A_n\times\A_n\to \{0,1\}$
	\begin{displaymath}
		\cfrak_n\left(\, (i,j)\, ,\, (i',j')\, \right) = \begin{cases}
			1 & i < i' < j < j'\text{ or } i' < i < j' < j \\
			0 & \text{otherwise}.
		\end{cases}
	\end{displaymath}
	We see $(\A_n,\cfrak_n)$ captures the combinatorics of arcs and their crossings on a polygon.
\end{example}

\begin{definition}\label{def:additive categorification}
	Let $(\A, \cfrak)$ be as in Definition~\ref{def:geometric model}.
	We construct a new additive category $\C$, called the \ul{additive categorification of $(\A,\cfrak)$}.
	The objects of $\C$ are the 0 object and finite direct sums of elements in $\A$.
	Let $\alpha,\beta,\gamma\in\A=\Ind(\C)$.	
	Define the Hom spaces and composition $\alpha\stackrel{f}{\to}\beta\stackrel{g}{\to}\gamma$, for $f$ and $g$ nonzero.
	\begin{align*}
		\Hom_{\C}(\alpha,\beta) &= \begin{cases}
 			\kk & \cfrak(\alpha,\beta)=1 \\
 			0 & \cfrak(\alpha,\beta)=0
 		\end{cases} &
 		& & & & g\circ f &= \begin{cases}
 			g\cdot f\in\kk & \alpha=\beta \text{ or } \beta=\gamma \\
 			0 & \text{otherwise}.
 		\end{cases}
	\end{align*}
	Extending bilinearly, we obtain a Krull--Schmidt category.
	
	Define a pairwise compatibility condition $\PP$, which we call the \ul{pairwise encoding of $\cfrak$}.
	We say two indecomposables $\alpha$ and $\beta$ are $\PP$-compatible if $\alpha=\beta$ or $\cfrak(\alpha,\beta)=0$.
\end{definition}

\begin{definition}\label{def:2D geometric model of a cluster theory}
	Let $(\A, \cfrak)$ be as in Definition~\ref{def:geometric model}, let $\C$ be the additive categorification of $(\A,\cfrak)$, and let $\PP$ the pairwise encoding of $\cfrak$.
	If $\PP$ induces the $\PP$-cluster theory of $\C$, we call $(\A,\cfrak)$ a \ul{2D geometric model of $\theory{\PP}{\C}$}.
	If $\theory{\QQ}{\D}$ is a cluster theory isomorphic to $\theory{\PP}{\C}$, we also say $(\A,\cfrak)$ is a 2D geometric model of $\theory{\QQ}{\D}$.
\end{definition}

\subsection{2D Geometric Models of Known Cluster Theories}\label{sec:geometric models for embeddings}
We begin with general crossing function that we use several times and a quadrilateral.
Then we define 2D geometric models of a cluster theories of polygons, the infinity-gon, and the completed infinity-gon from \cite{CCS06}, \cite{HJ12}, and \cite{BG18}, respectively.
We follow with 2D geometric models for a generalization of the infinity-gon from \cite{IT15a} and a generalization of the completed infinity-gon from \cite{PY21}.
The last 2D geometric model we define is ideal triangulations of the hyperbolic plane from \cite{IT15b}.

\begin{definition}\label{def:type A crossing}
	Let $\E$ be a totally ordered set and $\A\subset \E\times\E$ such that if $(i,j)\in\A$ then there exists $\ell\in\E$ such that $i<\ell<j$.
	The \ul{type $A$ crossing function on $\A$}, denoted $\cfrak:\A\times\A\to \{0,1\}$, is given by
	\begin{displaymath}
		\cfrak((i,j),(i',j')) = \begin{cases}
 			1 & ((i,j)= (i',j')) \text{ or } (i<i'<j<j')\text{ or } (i'<i<j'<j) \\
 			0 & \text{otherwise}.
 		\end{cases}
	\end{displaymath}
\end{definition}

\begin{definition}
	Let $\E$, $\A$, and $\cfrak$ be as in Definition~\ref{def:type A crossing}.
	A \ul{quadrilateral} in $(\A,\cfrak)$ is a set of four arcs using four points $i<j<k<\ell$ in $\E$: $(i,j)$, $(i,\ell)$, $(j,k)$, $(k,\ell)$.
\end{definition}

\begin{definition}[Polygons from \cite{CCS06}]\label{def:polygon}
	Let $(\A_n,\cfrak)$ be the same as in Example~\ref{xmp:polygon geometric model}.
	Let $\C(A_n)$ be the additive categorification of $(\A_n,\cfrak_n)$ and $\NN_n$ the pairwise encoding of $\cfrak_n$.
\end{definition}

It is shown in \cite{CCS06} that $\NN_n$ induces the $\NN_n$-cluster theory of $\C(A_n)$.

\begin{definition}[Infinity-gon from \cite{HJ12}]\label{def:infinity-gon}
	Let $\A_\infty = \{(i,j)\in\Z\times\Z \mid 2\leq j-i\}$ and let $\cfrak_\infty:\A_\infty\times\A_\infty\to\{0,1\}$ be the type $A$ crossing function on $\A_\infty$.
	Let $\C(A_\infty)$ be the additive categorification of $(\A_\infty,\cfrak_\infty)$ and $\NN_\infty$ the pairwise encoding of $\cfrak_\infty$.
\end{definition}

It is shown in \cite{HJ12} that $\NN_\infty$ induces the $\NN_\infty$-cluster theory of $\C(A_\infty)$.

\begin{definition}[Completed Infinity-gon from \cite{BG18}]\label{def:comp infinity-gon}\label{def:Cinftyclosed}\label{def:Ninftyclosed compatibility}
	Let $\E = \Z\cup\{-\infty,+\infty\}$. For all $i\in\Z$, we say $-\infty<i<+\infty$.
	Let 
	\begin{displaymath}
		\A = \left\{(i,j)\in\E\times\E \mid \exists \ell\in\Z \text{ s.t. } i<\ell<j \right\} \setminus \{(-\infty,+\infty)\}.
	\end{displaymath}
	Let $\cfrak_{\inftyclosed}:\A_{\inftyclosed}\times\A_{\inftyclosed}\to \{0,1\}$ be the type $A$ crossing function on $\A_{\inftyclosed}$.
	Let $\Cinftyclosed$ be the additive categorification of $(\A_{\inftyclosed},\cfrak_{\inftyclosed})$ and $\Ninftyclosed$ the pairwise encoding of $\cfrak_{\inftyclosed}$.
\end{definition}

It is shown in \cite{BG18} that $\Ninftyclosed$ incudes the $\Ninftyclosed$-cluster theory of $\Cinftyclosed$.

We obtain Definitions \ref{def:i-infinity-gon} and \ref{def:completed i-infinity-gon} by cutting the boundary of the circle and then taking a total order from one end of the cut to the other.
In \cite{IT15b}, the authors take triangulations of the disk by marking its boundary (the circle).
They take countably-many marked points with finitely-many accumulation points.
The accumulation points must be approached from both sides.
For this cluster theory (Definition~\ref{def:i-infinity-gon}) we cut at one of the accumulation points.

\begin{definition}[$i$-infinity-gon from \cite{IT15b}]
\label{def:i-infinity-gon}
	Fix a positive integer $i\geq 1$ and let $\E_{i,\infty} = \Z \times\{1,\ldots,i\}$ with total order $(m,\ell)\leq (n,p)$ if $\ell<p$ or if $\ell=p$ and $m\leq n$. 
	Let $\A_{i,\infty} = \{(j,\ell)\in\E_{i,\infty}\times\E_{i,\infty}\mid \exists k\in\E_{i,\infty}\text{ s.t. }j<k<\ell\}$.
	Let $\cfrak_{i,\infty}:\A_{i,\infty}\times \A_{i,\infty}\to \{0,1\}$ be the type $A$ crossing function on $\A_{i,\infty}$,\ \ $\C(A_{i,\infty})$ be the additive categorification of $(\A_{i,\infty},\cfrak_{i,\infty})$, and $\NN_{i,\infty}$ the pairwise encoding~of~$\cfrak_{i,\infty}$.
\end{definition}

It is shown in \cite{IT15a} that $\NN_{i,\infty}$ induces the $\NN_{i,\infty}$-cluster theory of $\C(A_{i,\infty})$.
Note that the 1-infinity-gon we have taken from \cite{IT15b}is the infinity-gon we have taken from \cite{HJ12}.

In \cite{PY21}, the authors take completed triangulations from \cite{IT15b} by also considering the accumulation points to be marked.
We cut at an accumulation point, which we consider to be the maximal element in our set of endpoints.

\begin{definition}[Completed $i$-inifnity-gon from \cite{PY21}]\label{def:completed i-infinity-gon}
	Fix a positive integer $i\geq 1$ and let $	\E_{\iinftyclosed} = (\Z\cup\{+\infty\}) \times\{1,\ldots,i\}$.
	We give $\Z\cup\{+\infty\}$ the usual total ordering and give $\E_{\iinftyclosed}$ the same order as in Definition~\ref{def:i-infinity-gon}.
	Let $\A_{\iinftyclosed} = \{(j,\ell)\in\E_{\iinftyclosed}\times \E_{\iinftyclosed} \mid \exists k\in\E_{\iinftyclosed} \text{ s.t. } j<k<\ell\}$.
	Let $\cfrak_{\iinftyclosed}: \A_{\iinftyclosed}\times\A_{\iinftyclosed}\to\{0,1\}$ be the type $A$ crossing function on $\A_{\iinftyclosed}$.
	Let $\C(A_{\iinftyclosed})$ be the additive categorification of $(\A_{\iinftyclosed},\cfrak_{\iinftyclosed})$ and $\NN_{\iinftyclosed}$ the pairwise encoding of $\cfrak_{\iinftyclosed}$.
\end{definition}

\noindent \textbf{!!} The completed 1-infinity-gon is \emph{not} the same as the completed infinity-gon.

{~}

It is shown in \cite{PY21} that $\NN_{\iinftyclosed}$ induces the $\NN_{\iinftyclosed}$-cluster theory of $\C(A_{\iinftyclosed})$.

Definition~\ref{def:hyperbolic} comes from the disk model of the hyperbolic plane.
We modify it to simplify calculations later.
\begin{definition}[Hyperbolic plane from \cite{IT15a}]\label{def:hyperbolic}
	Let $\E_\pi=(-\frac{\pi}{2},\frac{\pi}{2}]$ with total order inherited from $\R$.
	Let $\A_\pi=\{(x,y)\in\E_\pi\mid x<y\}$ and $\cfrak_\pi:\A_\pi\times\A_\pi\to\{0,1\}$ be the type $A$ crossing function on $\A_\pi$.
	Let $\C_\pi$ be the additive categorification of $(\A_\pi,\cfrak_\pi)$ and $\NN_\pi$ the pairwise encoding of $\cfrak_\pi$.
\end{definition}

It is shown in \cite{IT15a} that $\NN_\pi$ induces the $\NN_\pi$-cluster theory of $\C_\pi$. Notice our notation of $C_\pi$ instead of ``$\C(A_\pi)$.'' This is the notation used by Igusa and Todorov.

\subsection{Review of Other Cluster Theories}\label{sec:other theories}
We review cluster theories from \cite{IRT22b, KMMR21}.
The reader is referred to these works for additional details.

\begin{definition}[$\TT$-clusters from \cite{KMMR21}]\label{def:T-clusters}
	Let $\Zcal=\{(x,x) \mid -\frac{\pi}{2}<x<\frac{\pi}{2}\}$.
	Let $\Zcal[1]=\{(x+\pi,-x)\in\R\times(-\frac{\pi}{2},\frac{\pi}{2})\mid (x,x)\in\Zcal\}$.
	Let
	\begin{displaymath}
		\left. \C_{\Zcal} = \left\{ 
			(x,y)\in\R\times\left(-\frac{\pi}{2},\frac{\pi}{2}\right)\, \, \right|\, \,
			\exists (x_0,y)\in\Zcal,\,  \exists (x_1,y)\in\Zcal[1] \text{ s.t. } x_0\leq x \leq x_1
		\right\}.
	\end{displaymath}
	We say $(x,y)$ and $(z,w)$ in $\Ind(\C_\Zcal)$ are \emph{not} $\TT$-compatible if and only if there exists a rectangle inside $\C_{\Zcal}\cup\{(\frac{\pi}{2},\frac{\pi}{2})\}\cup\{(x,-\frac{\pi}{2})\mid -\frac{\pi}{2}\leq x \leq \frac{3\pi}{2}\}$ with sides sloped at $\pm 1$ whose left and right corners are $\{(x,y),(z,w)\}$.
\end{definition}

It is shown in \cite{KMMR21} that $\TT$ incudes the $\TT$-cluster theory of $\C_{\Zcal}$. Notice again we are using different notation, this time from \cite{KMMR21}.
The set $\Ind(\C_\Zcal)$ is a triangle in $\R^2$ without $(\frac{\pi}{2},\frac{\pi}{2})$ and without its bottom side $\{(x,-\frac{\pi}{2})\mid -\frac{\pi}{2}\leq x \leq \frac{3\pi}{2}\}$.

For Definition~\ref{def:E-clusters}, we need a bit of background.
We refer the reader to \cite{IRT22a, IRT22b} for more details.
Consider $\R$ as a category where $\Hom_\R(x,y) = \emptyset$ if $y>x$ and $\Hom_\R(x,y) = \{*\}$ if $x\geq y$.
A \ul{pointwise finite-dimensional representation} $V$ of $\R$ is a functor from $\R$ to finite-dimensional vector spaces.
This yields an abelian category $\repAR$ of finitely-generated representations and its bounded derived category $\DbAR$.
From here we take an orbit and obtain the triangulated category $\CAR$ where $\Ind(\CAR)=\Ind(\repAR)$.

There are 2 indecomposable projective representations $P_x$ and $P_{x)}$, for each $x\in\R$.
There is also a projective $P_{+\infty)}$.
There is a monomorphism $P_y\to P_x$ if and only if $x>y$ and a nonzero morphism $P_{x)}\to P_x$ (but not the other way) for each $x\in\R$.
For any indecomposable projective $P$ there is a nonzero morphism $P\to P_{+\infty)}$ (but not the other way).
For each pair $a<b$ in $\R$ there are 4 indecomposable representations that have the following projective resolutions:
\begin{align*}
	P_{a)}\to P_{b)} &\to M_{[a,b)} & P_a \to P_b &\to M_{(a,b]}\\
	P_{a)}\to P_b &\to M_{[a,b]} & P_a\to P_{b)} &\to M_{(a,b)}.
\end{align*}
$I_{(-\infty}=P_{+\infty)}$ and other indecomposable injectives have projective resolutions:
\begin{align*}
	P_{x)} \to P_{+\infty)} &\to I_x & P_x \to P_{+\infty)}\to I_{(x}.
\end{align*}
Each indecomposable $M_{|a,b|}$ (see `$|$'s in Conventions) has its \ul{g-vector} $[P]-[P']$ in $K^{\text{split}}_0(\CAR)$ where $P'\to P\to M_{|a,b|}$ is the projective resolution of $M_{|a,b|}$.
For $[V]=\sum_i m_i[V_i]$ and $[W]=\sum_j n_j[W_j]$ in $\KsplitCARS$, define an Euler form:
\begin{displaymath}
	\left\langle [V]\, ,\, [W] \right\rangle :=
	\sum_i \sum_j (\dim_\kk\Hom_{\CAR}(V_i,W_j)).
\end{displaymath}

\begin{definition}[$\EE$-clusters from \cite{IRT22b}]\label{def:E-clusters}
	Let $\CAR$ be the category defined above.
	Let $V,W\in\Ind(\C)$ with $g$-vectors $[P_V]-[P'_V]$ and $[P_W]-[P'_W]$.
	We say $\{V,W\}$ is $\EE$-compatible if and only if
	\begin{displaymath}
		\langle [P_V]-[P'_V]\, ,\, [P_W]-[P'_W] \rangle \geq 0
		\qquad \text{and} \qquad
		\langle [P_W]-[P'_W]\, ,\, [P_V]-[P'_V] \rangle \geq 0.
	\end{displaymath}
\end{definition}

It is shown in \cite{IRT22b} that $\EE$ induces the $\EE$-cluster theory of $\CAR$.

\subsection{Abstract Cluster Structures}\label{sec:abstract cluster structures}
In this section we define abstract cluster structures (Definition~\ref{def:cluster structure}) and show examples from \cite{BMRRT06,CCS06,BIRS09,HJ12,IT15a,IT15b} that we use in Section~\ref{sec:strutures in theories}.
We then we define an embedding of cluster structures (Definition~\ref{def:restrict to cluster structure}).

\begin{definition}\label{def:cluster structure}
Let $\C$ be a Krull--Schmidt cateogory and let $\PP$ be a pairwise compatibility condition that induces $\theory{\PP}{\C}$.
Let $\mathscr S_{\PP}(\C)\subset \mathscr T_{\PP}(\C)$ be a subcategory.
We call $\mathscr S_{\PP}(\C)$ a \ul{$\PP$-cluster structure of $\C$} if it satisfies the following conditions.
\begin{itemize}
\item For each $\PP$-cluster $T$ in $\mathscr S_{\PP}(\C)$, every $X\in T$ is $\PP$-mutable and its $\PP$-mutation $(T\setminus \{X\})\cup\{Y\}$ is also in $\mathscr S_{\PP}(\C)$.
\item For each $\PP$-cluster $T$ in $\mathscr S_{\PP}(\C)$ the category $\C / \add T$ is abelian.
\end{itemize}
We define $I_{\PP,\C}^{\mathscr{S}}:\struc{\PP}{\C}\to \mathcal Set$ to be the composition of the inclusion $\struc{\PP}{\C}\to \theory{\PP}{\C}$ and $I_{\PP,\C}$.
\end{definition}
\noindent\textbf{!!} Notice $\struc{\PP}{\C}$ is not necessarily unique.

\begin{example}\label{xmp:An structure}
	The cluster theory $\theory{\NN_n}{\C(A_n)}$ (Definition~\ref{def:polygon}) from \cite{CCS06} is also an $\NN_n$-cluster structure of $\C(A_n)$.
\end{example}

\begin{example}\label{xmp:cluster category}
	In general, any cluster category as in \cite{BMRRT06,BIRS09} induces a cluster theory that is also a cluster structure as in Definition~\ref{def:cluster structure}.
\end{example}

\begin{example}\label{xmp:Ainfty structure}
In \cite{HJ12} the authors describe the $\NN_{1,\infty}$-cluster structure of $\C(A_{1,\infty})$ as those triangulations of the infinity-gon that are functorially finite.
In their setting, those triangulations $T$ which are not functorially finite do not yield abelian quotient categories.
That is, the cluster structure in \cite{HJ12} is a cluster structure as in Definition~\ref{def:cluster structure}.
Denote this cluster structure by $\struc{\NN_{1,\infty}}{C(A_{1,\infty})}$.
\end{example}

\begin{definition}\label{def:restrict to cluster structure}
Let $\C$ and $\D$ be Krull--Schmidt categories with pairwise compatibility conditions $\PP$ and $\QQ$, respectively, that induce $\theory{\PP}{\C}$ and $\theory{\QQ}{\D}$.
Let $\struc{\PP}{\C}$ and $\struc{\QQ}{\D}$ be a $\PP$-cluster structure of $\C$ and a $\QQ$-cluster structure of $\D$, respectively.
A pair $F:\struc{\PP}{\C}\to \struc{\QQ}{\D}$ and $\eta:I_{\PP,\C}^{\mathscr{S}}\to (I_{\QQ,\D}^{\mathscr{S}}\circ F)$ is an \ul{embedding of cluster structures} if $F$ is faithful and injective on objects and if $\eta_T$ is an injection for each $\PP$-cluster $T$.
\end{definition}

See Example~\ref{xmp:Ainftyclosed not-structure} in Section~\ref{sec:strutures in theories}.

\section{Proof of Theorem~\ref{thm:intro:theories}}
\label{sec:embeddings of cluster theories}\label{sec:chain of embeddings}
In this section we prove Theorem~\ref{thm:intro:theories} by proving there is a diagram of embeddings of type $A$ cluster theories as in Figure~\ref{fig:big diagram}.
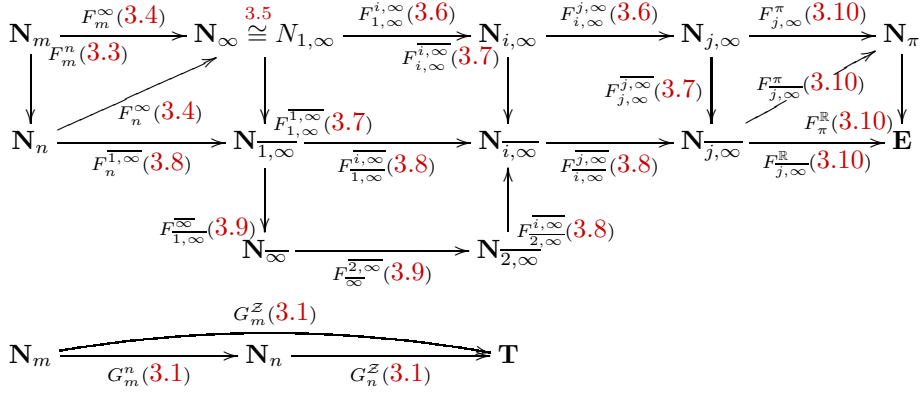
\begin{figure}
	\begin{displaymath}
	\xymatrix@C=11ex{
		\NN_m \ar[r]^-{F_m^\infty(\ref{sec:embeddings:finite-infinite})} \ar[d]^<{\ F_m^n(\ref{sec:embeddings:finite})}  \ar@{}[rr]^-{\begin{matrix}\text{\scriptsize\ref{sec:embeddings:isomorphism}\ \ }\end{matrix}} &
		\NN_\infty\cong N_{1,\infty} \ar[r]^-{F_{1,\infty}^{i,\infty}(\ref{sec:embeddings:i-infinity-gons})} \ar[d]^>{F_{1,\infty}^{\overline{1,\infty}}(\ref{sec:embeddings:completing i-infinity-gons})} &
		\NN_{i,\infty} \ar[r]^-{F_{i,\infty}^{j,\infty}(\ref{sec:embeddings:i-infinity-gons})} \ar[d]_<{F_{i,\infty}^{\iinftyclosed}(\ref{sec:embeddings:completing i-infinity-gons})} &
		\NN_{j,\infty} \ar[r]^-{F_{j,\infty}^\pi (\ref{sec:embeddings:continuous})} \ar[d]_-{F_{j,\infty}^{\jinftyclosed}(\ref{sec:embeddings:completing i-infinity-gons})} &
		\NN_\pi \ar[d]_>{F_\pi^\R (\ref{sec:embeddings:continuous})}
		\\
		\NN_n \ar[r]_-{F_n^{\overline{1,\infty}}(\ref{sec:emeddings:complete i-infinity-gon})} \ar[ur]_-{F_n^\infty(\ref{sec:embeddings:finite-infinite})} &
		\NN_{\overline{1,\infty}} \ar[r]_{F_{\overline{1,\infty}}^{\iinftyclosed}(\ref{sec:emeddings:complete i-infinity-gon})} \ar[d]_>{F_{\overline{1,\infty}}^{\inftyclosed}(\ref{sec:embeddings:2-infinity-gon})} &
		\NN_{\iinftyclosed} \ar[r]_-{F_{\iinftyclosed}^{\jinftyclosed}(\ref{sec:emeddings:complete i-infinity-gon})} &
		\NN_{\jinftyclosed} \ar[r]_-{F_{\jinftyclosed}^\R (\ref{sec:embeddings:continuous})} \ar[ur]|-{F_{\jinftyclosed}^\pi (\ref{sec:embeddings:continuous})} &
		\EE
		\\
		&
		\NN_{\inftyclosed} \ar[r]_-{F_{\inftyclosed}^{\overline{2,\infty}}(\ref{sec:embeddings:2-infinity-gon})} &
		\NN_{\overline{2,\infty}} \ar[u]_<{F_{\overline{2,\infty}}^{\iinftyclosed}(\ref{sec:emeddings:complete i-infinity-gon})}
		\\
		\NN_m \ar[r]_-{G_m^n (\ref{sec:embeddings:G})} \ar@/^2ex/[rr]^-{G_m^\Zcal (\ref{sec:embeddings:G})} &
		\NN_n \ar[r]_-{G_n^\Zcal (\ref{sec:embeddings:G})} &
		\TT
	}
	\end{displaymath}
	\caption{Embeddings of cluster theories of type $A$ for any integers $m,n,i,j$ satisfying $2\leq m<n$ and $2\leq i < j$.
		For each cluster theory $\theory{\PP}{\C}$, we have only written its pairwise compatibility condition $\PP$ in order to save space.
		We have also written each embedding $(F,\eta)$ as $F$ with the section in which it is defined.}\label{fig:big diagram}
\end{figure}

\subsection{The $G_*^*$ Embeddings}\label{sec:embeddings:G}
For $2\leq m<n$, the embeddings of cluster theories $G_m^n:\theory{\NN_m}{\C(A_m)}\to \theory{\NN_n}{\C(A_n)}$ and $G_m^{\Zcal}:\theory{\NN_m}{\C(A_m)}\to \theory{\TT}{\C_{\Zcal}}$ are defined in \cite[Section~6.4.2]{KMMR21} with natural transformations $\xi_n^m$ and $\xi_m^\Zcal$, respectively.
Furthermore, in \cite[Theorem~6.18]{KMMR21} it is shown that $G_m^{\Zcal} = G_n^{\Zcal}\circ G_m^n$.

\subsection{The general technique}\label{sec:embeddings:general}
We will use the same technique several times, which we summarize in the following two lemmas.

\begin{lemma}\label{lem:general technique 1}
	Let $\theory{\PP}{\C}$ and $\theory{\QQ}{\D}$ be cluster theories, $\Phi:\Ind(\C)\to\Ind(\D)$ be an injection, and $F:\theory{\PP}{\C}\to\theory{\QQ}{\D}$ be a functor. 
	Suppose $\{X,Y\}$ is $\PP$-compatible if and only if $\{\Phi(X),\Phi(Y)\}$ is $\QQ$-compatible and, for each $\PP$-cluster $T$, if $X\in T$ then $\Phi(X)\in F(T)$.
	Then there is an embedding of cluster theories $(F,\eta):\theory{\PP}{\C}\to\theory{\QQ}{\D}$ such that $\eta_T(X) = \Phi(X)$, for all $X$ in each $\PP$-cluster $T$.
\end{lemma}
\begin{proof}
	Recall that cluster theories are small categories.
	So, we need only to check that $F$ is injective on objects.
	Suppose $T$ and $T'$ are distinct clusters in $\mathscr T_{\PP}(\C)$.
	Then there exists $Y\in T'$ and $X\in T$ such that $\{X,Y\}$ is not $\PP$-compatible.
	By assumption, $\{\Phi(X),\Phi(Y)\}$ is not $\QQ$-compatible.
	Since $\Phi(X)\in F(T)$ and $\Phi(Y)\in F(T')$, we must have $F(T)\neq F(T')$.
\end{proof}
Exploiting notation, we write $\Phi(T)$ to mean $\{\Phi(X)\mid X\in T\}$.

\begin{lemma}\label{lem:general technique 2}
	Let $\E$ and $\E'$ be totally ordered sets.
	Let $\A\subset\E\times\E$ and $\A'\subset\E'\times\E'$ be as in Definition~\ref{def:type A crossing}.
	Let $\cfrak$ and $\cfrak'$ be the type $A$ crossing functions on $\A$ and $\A'$, respectively.
	Let $\C$ and $\mathcal D$ be the additive categorifications of $(\A,\cfrak)$ and $(\A',\cfrak')$, respectively.
	Let $\PP$ and $\QQ$ be the pairwise encodings of $\cfrak$ and $\cfrak'$, respectively.
	
	Any order-preserving injection $\phi:\E\hookrightarrow \E'$ induces an injection $\Phi:\Ind(\C)\to\Ind(\mathcal D)$ such that $\{\alpha,\beta\}$ is $\PP$-compatible if and only if $\{\Phi(\alpha),\Phi(\beta)\}$ is $\QQ$-compatible.
\end{lemma}
\begin{proof}
	For this proof, let $(i,j), (i',j')\in \A$ such that $(i,j)\neq (i',j')$.
	Immediately, either $i\neq i'$ or $j\neq j'$.
	In either case, $(\phi(i),\phi(j))\neq (\phi(i'),\phi(j'))$ and so $\Phi$ is injective.
	
	Now suppose $\{(i,j),(i',j')\}$ is not $\PP$ compatible.
	Then, up to symmetry, $i<i'<j<j'$.
	Since $\phi$ is an order preserving injection, we have $\phi(i)<\phi(i')<\phi(j)<\phi(j')$, which means $\{\Phi((i,j)),\Phi((i',j'))\}$ is not $\QQ$-compatible.
	If we instead assume $\{\Phi((i,j)),\Phi((i',j'))\}$ is not $\QQ$-compatible we reverse the argument and see $\{(i,j),(i',j')\}$ is not $\PP$-compatible.
\end{proof}

\subsection{Finite Embeddings}\label{sec:embeddings:finite}
We begin with $(F_m^n, \eta_m^n)$ by defining $(F_m^{m+1},\eta_m^{m+1})$.
For $m\geq 2$, let $\phi:\E_m\to \E_{m+1}$ be given by $\phi(i)=i$ if $m$ is odd and $\phi(i)=i+1$ if $m$ is even.
By Lemma~\ref{lem:general technique 2}, we have an injection $\Phi:\A_m\to A_{m+1}$ such that $\{\alpha,\beta\}$ is $\NN_m$-compatible if and only if $\{\Phi(\alpha),\Phi(\beta)\}$ is $\NN_{m+1}$-compatible.
Recall our notation $\Phi(T)$ in Section~\ref{sec:embeddings:general}.
For an $\NN_m$-cluster $T$ define $F_m^{m+1}(T)$ by
\begin{displaymath}
	F_m^{m+1}(T) := \begin{cases}
		\{\Phi(T) \cup \{(1,m)\} & m \text{ odd} \\
		\{\Phi(T) \cup \{(2,m+1)\} & m \text{ even}.
	\end{cases}
\end{displaymath}

The definition of $F_m^{m+1}$ yields an embdding of cluster theories $(F_m^{m+1},\eta_m^{m+1}):\theory{\NN_m}{\C(A_m)}\to \theory{\NN_{m+1}}{\C(A_{m+1})}$.
It is immediate that $F_m^{m+1}(T)$ is an $\NN_{m+1}$-cluster.
Let $T\to T'$ be an $\NN_m$-mutation where we exchange $(i,j)$ and $(k,\ell)$.
Without loss of generality: $i<k<j<\ell$.
By \cite[Lemma~3.1]{CCS06}, we have a quadrilateral in $(\A_m,\cfrak_m)$ with sides $(i,\ell)$, $(i,k)$, $(k,j)$, $(j,\ell)$ whose diagonals are $(i,j)$ and $(k,\ell)$.
The image under $\Phi$ of the quadrilateral and its diagonals is also a quadrilateral and its diagonals.
Thus, $F_m^{m+1}(T)\to F_m^{m+1}(T')$ is a $\NN_{m+1}$-mutation.
By Lemma~\ref{lem:general technique 1} we have an embedding of cluster theories $(F_m^{m+1},\eta_m^{m+1}):\theory{\NN_m}{\C(A_m)}\to \theory{\NN_{m+1}}{\C(A_{m+1})}$.

Define $(F_m^n,\eta_m^n):= (F_{n-1}^n,\eta_{n-1}^n)\circ\cdots \circ (F_{m+1}^{m+2},\eta_{m+1}^{m+2})\circ (F_m^{m+1},\eta_m^{m+1})$.

\subsection{Finite-infinite Embeddings}\label{sec:embeddings:finite-infinite}
We now define $(F_n^\infty,\eta_n^\infty)$.
Recall that if $n>2$ is odd, $\lceil \frac{n}{2}\rceil = \frac{n+1}{2}$.
For $n\geq 2$, let $\phi:\E_n\to\E_\infty$ be given by $\phi(i) = i-\lceil\frac{n}{2}\rceil$.
By Lemma~\ref{lem:general technique 2} we have $\Phi:\A_n\to\A_\infty$ such that $\{\alpha,\beta\}$ is $\NN_n$-compatible if and only $\{\Phi(\alpha),\Phi(\beta)\}$ is $\NN_\infty$-compatible.
Let $T$ be an $\NN_n$-cluster and define $F_n^\infty(T)$ by
\begin{displaymath}
	F_n^\infty(T) :=
	\Phi(T) \cup\{(\phi(1),\phi(n))\} \cup
	\left\{(-i,i), (-i,i+1)\ \left|\ i\geq \left\lceil \frac{n}{2}\right\rceil,\ i\text{ odd}\right. \right\}.
\end{displaymath}

The definition of $F_n^\infty$ yields an embedding of cluster theories $(F_n^\infty,\eta_n^\infty)$ from $\theory{\NN_n}{\C(A_n)}$ to $\theory{\NN_\infty}{\C(A_\infty)}$.
The proof is similar to the proof for $(F_m^{m+1},\eta_m^{m+1})$.

For $2\leq m < n$, we see $(F_m^\infty,\eta_m^\infty) = (F_n^\infty,\eta_n^\infty)\circ (F_m^n,\eta_m^n)$.
It is sufficient to prove the statement when $n=m+1$, which is left to the reader.

\subsection{The isomorphism}\label{sec:embeddings:isomorphism}
We now prove $\theory{\NN_\infty}{\C(A_\infty)}\cong \theory{\NN_\infty}{\C(A_{1,\infty})}$.
We have an order-preserving bijection $\phi:\E_\infty\to \E_{1,\infty}$ given by $\phi(i) = (i,1)$.
This induces a bijection on arcs $\Phi:\A_\infty\to\A_{1,\infty}$ such that $(\alpha,\beta)$ is $\NN_\infty$-compatible if and only if $\{\Phi(\alpha),\Phi(\beta)\}$ is $\NN_{1,\infty}$-compatible.
By \cite[Lemma~2.0.3]{R20}, we have an isomorphism of cluster theories.

\subsection{Embedding $i$-infinity-gons}\label{sec:embeddings:i-infinity-gons}
We now define $(F_{i,\infty}^{j,\infty},\eta_{i,\infty}^{j,\infty})$ for $2\leq i < j$.
Let $\phi:\E_{i,\infty}\to E_{j,\infty}$ be given by $\phi(k,\ell)= (k,\ell+(j-i))$.
By Lemma~\ref{lem:general technique 2} we have an order-preserving injection which yields an injection $\Phi:\A_{i,\infty}\to A_{j,\infty}$ that preserves compatibility.
Let $T$ be an $\NN_{i,\infty}$-cluster and define $F_{i,\infty}^{j,\infty}(T)$ by
\begin{displaymath}
	F_{i,\infty}^{j,\infty}(T) := \Phi(T) \cup \{((-j,\ell),(j,\ell)),\ ((-j,\ell),(j+1,\ell)) \mid 0<j\text{ odd},\ 0<\ell\leq j-i \}.
\end{displaymath}
It is straightforward to check that $F_{i,\infty}^{j,\infty}(T)$ is an $\NN_{j,\infty}$-cluster.
By \cite[Theorem~2.4.1]{IT15b}, mutation is given by exchanging diagonals in quadrilaterals.
We see $\Phi$ preserves quadrilaterals and their diagonals.
By Lemma~\ref{lem:general technique 1} we have an embedding of cluster theories $(F_{i,\infty}^{j,\infty},\eta_{i,\infty}^{j,\infty}):\theory{\NN_{i,\infty}}{\C(A_{i,\infty})}\to\theory{\NN_{j,\infty}}{\C(A_{j,\infty})}$.

\subsection{Embedding completed $i$-infinity-gons}\label{sec:embeddings:completing i-infinity-gons}
We define $(F_{\iinftyclosed}^{\jinftyclosed},\eta_{\iinftyclosed}^{\jinftyclosed})$ in a similar way to$(F_{i,\infty}^{j,\infty},\eta_{i,\infty}^{j,\infty})$ by using a similar $\phi$ and $\Phi$.
The notable difference is $F_{\iinftyclosed}^{\jinftyclosed}(T)$,
\begin{align*}
	F_{\iinftyclosed}^{\jinftyclosed}(T) :=&\  \Phi(T) \cup \{((-j,\ell),(j,\ell)),\ ((-j,\ell),(j+1,\ell)) \mid 0<j\text{ odd},\ 0<\ell\leq j-i \} & \\
		&\ \cup\{((+\infty,\ell),(+\infty,j))\mid 0<\ell\leq j-i\},
\end{align*}
where we need the extra arcs on the second line.
The proof of the embedding of cluster theories is essentially the same, using \cite[Theorem~6.13]{PY21} instead of \cite[Theorem~2.4.1]{IT15b}.

\subsection{Completing $i$-infinity-gons}\label{sec:emeddings:complete i-infinity-gon}
We now define $(F_{i,\infty}^{\iinftyclosed},\eta_{i,\infty}^{\iinftyclosed})$.
Let $\phi:\E_{i,\infty}\to\E_{\iinftyclosed}$ be given by $\phi(k,\ell)= (k,\ell)$.
By Lemma~\ref{lem:general technique 2} we have an injection $\Phi:\A_{i,\infty}\to \A_{\iinftyclosed}$ such that compatibility is preserved.
We generalize Pr\"ufer-completions from \cite[Definition~5.2]{BG18}.
Let $X$ be an $\NN_{\iinftyclosed}$-compatible set of arcs.
For $1\leq n\leq i$, define $\Prufer_n(X)$ by
\begin{displaymath}
	\Prufer_n(X):= X\cup \{((j,k),(+\infty,n))\mid X\cup\{((j,k),(+\infty,n))\}\text{ is }\NN_{\iinftyclosed}\text{-compatible}\}.
\end{displaymath}
Then, for an $\NN_{i,\infty}$-cluster $T$, we define $F_{i,\infty}^{\iinftyclosed}(T)$ by
\begin{displaymath}
	F_{i,\infty}^{\iinftyclosed}(T) := \Prufer_1(\Prufer_2(\cdots \Prufer_{i-1}(\Prufer_i(\Phi(T))))).
\end{displaymath}
Since the only arcs ``missing'' from $\Phi(T)$ are those to the $(+\infty,\ell)$ points, we see $F_{i,\infty}^{\iinftyclosed}(T)$ is an $\NN_{\iinftyclosed}$-cluster by construction.
Notice also that
\begin{displaymath}
(F_{i+1,\infty}^{\overline{i+1,\infty}},\eta_{i+1,\infty}^{\overline{i+1,\infty}})\circ(F_{i,\infty}^{i+1,\infty},\eta_{i,\infty}^{i+1,\infty})=(F_{\iinftyclosed}^{\overline{i+1,\infty}},\eta_{\iinftyclosed}^{\overline{i+1,\infty}})\circ(F_{i,\infty}^{\iinftyclosed},\eta_{i,\infty}^{\iinftyclosed}).
\end{displaymath}
Thus the equation holds when we replace $i+1$ with $j>i$. 

Define $(F_n^{\overline{1,\infty}},\eta_n^{\overline{1,\infty}}) := (F_{1,\infty}^{\overline{1,\infty}},\eta_{1,\infty}^{\overline{1,\infty}})\circ (F_n^{1,\infty},\eta_n^{1,\infty})$.

\subsection{Completed infinity-gon and completed 2-infinity-gon}\label{sec:embeddings:2-infinity-gon}
We first define $(F_{\overline{1,\infty}}^{\overline{\infty}},\eta_{\overline{1,\infty}}^{\overline{\infty}})$.
Let $\phi_1:\E_{\overline{1,\infty}}\to\E_{\overline{\infty}}$ be given by $\phi_1(i,1)= i$.
By Lemma~\ref{lem:general technique 2} and obtain $\Phi_1:\A_{\overline{1,\infty}}\to\A_{\overline{\infty}}$ that preserves compatibility.
For a set of $\NN_{\overline{\infty}}$-compatible arcs $X$, the \ul{adic completion} of $X$ \cite[Definition~5.2]{BG18}, which we denote $\adic(X)$, is given by
\begin{displaymath}
	\adic(X):= X\cup \{(-\infty,i) \mid X\cup\{(-\infty,i)\}\text{ is }\NN_{\overline{\infty}}\text{-compatible}\}.
\end{displaymath}
Let $T$ be an $\NN_{\overline{1,\infty}}$-cluster and define $F_{\overline{1,\infty}}^{\overline{\infty}}(T)$ by $F_{\overline{1,\infty}}^{\overline{\infty}}(T) := \adic(\Phi_1(T))$.
By \cite[Theorem~5.4]{BG18}, $F_{\overline{1,\infty}}^{\overline{\infty}}(T)$ is an $\NN_{\overline{\infty}}$-cluster.
By \cite[Remark 2.2]{BG18} and \cite[Theorem~6.13]{PY21}, both $\Ninftyclosed$-mutation and $\NN_{\overline{1,\infty}}$-mutation are given by exchanging diagonals in a quadrilateral.
By Lemma~\ref{lem:general technique 1} we have an embedding of cluster theories 
$(F_{\overline{1,\infty}}^{\overline{\infty}},\eta_{\overline{1,\infty}}^{\overline{\infty}}):\theory{\NN_{\overline{1,\infty}}}{\C(A_{\overline{1,\infty}})}\to\theory{\Ninftyclosed}{\Cinftyclosed}$.

Now we define $(F_{\overline{\infty}}^{\overline{2,\infty}},\eta_{\overline{\infty}}^{\overline{2,\infty}})$.
Let $\phi_2:\E_{\overline{\infty}}\to \E_{\overline{2,\infty}}$ be given by $\phi_2(i) = (i,2)$ if $i>-\infty$ and $\phi_2(-\infty)=(+\infty,1)$.
Again by Lemma~\ref{lem:general technique 2} we have $\Phi_2:\A_{\overline{\infty}}\to\A_{\overline{2,\infty}}$ that preserves compatibility.
For an $\NN_{\overline{\infty}}$-cluster $T$, define $F_{\inftyclosed}^{\overline{2,\infty}}(T)$ by
\begin{displaymath}
	F_{\inftyclosed}^{\overline{2,\infty}}(T) :=
	\Phi_2(T) \cup \{((-i,1),(i,1)),\ ((-i,1),(i+1,1))\mid i\text{ odd}\}.
\end{displaymath}
Again by \cite[Remark 2.2]{BG18} and \cite[Theorem~6.13]{PY21}, both $\Ninftyclosed$-mutation and $\NN_{\overline{2,\infty}}$-mutation are given by exchanging diagonals in a quadrilateral.
Therefore, by Lemma~\ref{lem:general technique 1}, we have an embedding of cluster theories $(F_{\overline{\infty}}^{\overline{2,\infty}},\eta_{\overline{\infty}}^{\overline{2,\infty}}):\theory{\Ninftyclosed}{\Cinftyclosed}\to\theory{\NN_{{\overline{2,\infty}}}}{\C(A_{\overline{2,\infty}})}$.
By definition, for $i\geq 2$, we have $(F_{\overline{1,\infty}}^{\iinftyclosed},\eta_{\overline{1,\infty}}^{\iinftyclosed})=(F_{\overline{2,\infty}}^{\iinftyclosed},\eta_{\overline{2,\infty}}^{\iinftyclosed})\circ(F_{\overline{\infty}}^{\overline{2,\infty}},\eta_{\overline{\infty}}^{\overline{2,\infty}})\circ(F_{\overline{1,\infty}}^{\overline{\infty}},\eta_{\overline{1,\infty}}^{\overline{\infty}})$.

\subsection{Embeddings to $\NN_\pi$-clusters and $\EE$-clusters}\label{sec:embeddings:continuous}
By \cite[Theorem 5.2.8]{IRT22b} there is an embedding of cluster theories $(F_\R^\pi,\eta_\R^\pi):\theory{\NN_\pi}{\C_\pi}\to \theory{\EE}{\CAR}$.

We now define $F_{\iinftyclosed}^\pi$.
Extend $\tan$ by setting $\tan(\frac{\pi}{2})=+\infty$.
For each integer $\ell\in\Z$, let $\{a_i^\ell\}_{-\infty}^{+\infty}$ be a monotonic sequence such that $\lim\limits_{i\to-\infty}(a_i^\ell)=\ell$ and $\lim\limits_{i\to+\infty}(a_i^\ell)=\ell+1$.
For each $i,j,k,\ell\in\Z$ such that $\ell<0$, $0\leq k$, and $0\leq j \leq 2^k$, define $a_{i,j,k}^\ell := a_i^\ell + \left(\frac{j}{2^k}\right)(a_{i+1} - a_i)$.
Define $a_{+\infty}^\ell=a_{-\infty}^{\ell+1}=\ell+1$.

Now fix $i\geq 1$ and define $\NN_\pi$-compatible sets $A$, $B$, and $C$ in $\C_\pi$:
\begin{align*}
A &:= \left\{(\tan^{-1}a_{i,j,k}^\ell,\  \tan^{-1}a_{i,j+1,k}^\ell) \mid i,j,k,\ell\in\Z, \ 0\leq k,\ 0\leq j < 2^k\right\} \\
B &:= \left\{\left.\left(\tan^{-1}a_i^\ell,\frac{\pi}{2}\right)\, \right|\, i,\ell\in\Z,\ \ell\geq 0 \right\}\cup\left\{\left.\left( \tan^{-1} \ell, \frac{\pi}{2} \right)\ \right|\ \ell\in\Z,\ \ell\geq 0 \right\} \\
C &:= \left\{ \left. (\tan^{-1}a_m^\ell, 0)\ \right| m,\ell\in\Z,\ \ell< -i\right\} \cup \left\{\left.\left( \tan^{-1} \ell,\ 0 \right)\ \right|\ \ell\in\Z,\ \ell\leq-i \right\}.
\end{align*}
Let $\phi:\E_{\iinftyclosed}\to \E_\pi$ be given by $(m,\ell+i+1)\mapsto \tan^{-1}a_m^\ell$.
This is an order-preserving injection and so by Lemma~\ref{lem:general technique 2} we have an inclusion $\Phi:\A_{\iinftyclosed}\to \A_\pi$ such that $\{\alpha,\beta\}$ is $\NN_{\iinftyclosed}$-compatible if and only if $\{\Phi(\alpha),\Phi(\beta)\}$ is $\NN_\pi$-compatible.
Let $T$ be a $\NN_{\iinftyclosed}$-cluster and define $F_{\iinftyclosed}^\pi(T)$:
\begin{displaymath}
	F_{\iinftyclosed}^\pi(T) := \Phi(T) \cup A \cup B \cup C.
\end{displaymath}

The definition of $F_{\iinftyclosed}^\pi$ has a natural completion to a functor which induces an embedding of cluster theories $F_{\iinftyclosed}^\pi:\theory{\NN_{\iinftyclosed}}{\C(A_{\iinftyclosed})}\to \theory{\NN_\pi}{\C_\pi}$.
We first prove that for any $\NN_{\iinftyclosed}$-cluster $T$, $F_{\iinftyclosed}^\pi(T)$ is an $\NN_\pi$-cluster and then prove that $F_{\iinftyclosed}^\pi$ respects both mutations.
By Lemma~\ref{lem:general technique 1}, this is sufficient to prove the proposition.

Let $T$ be a $\NN_{\iinftyclosed}$-cluster and $(x,y)\in\A_\pi$ such that $\{(x,y)\}\cup F_{\iinftyclosed}^\pi(T)$ is $\NN_\pi$-compatible.
If $x\neq \tan^{-1}a_{m,j,k}^\ell$, for any $m,j,k,\ell$ as above, then there must be some pair $(\tan^{-1}a_{m,j,k}^\ell,\tan^{-1}a_{m,j+1,k}^\ell)$ in $\A_\pi$ that is not $\NN_\pi$-compatible with $(x,y)$.
By a similar argument for $y$, we have $x=\tan^{-1}a_{m,j,k}^\ell$ and $y=\tan^{-1} a_{m',j',k'}^{\ell'}$ as above.

If $\ell=\ell'$, $i\neq i'$, and either $j\neq 0$ or $j'\neq 0$, then $(x,y)$ would not be $\NN_\pi$-compatible with $\{\tan^{-1}a_m^\ell, \tan^{-1}a_{m+1}^\ell\}$.
So, if $\ell=\ell'$ and $i\neq i'$, then $j=0$ and $j'=0$.
If $\ell\neq\ell'$ and either $j\neq 0$ or $j'\neq 0$ then $(x,y)$ is not $\NN_\pi$-compatible with $(\tan^{-1}a_m^\ell,\tan^{-1}a_{m+1}^\ell)$ or with $(\tan^{-1}a_{m'}^{\ell'},\tan^{-1}a_{m'+1}^{\ell'})$.
So, if $\ell\neq\ell'$ then we must have $x=a_m^\ell$ and $y=a_{m'}^{\ell'}$.
Thus, if $j\neq 0$ and $j'\neq 0$ then $(x,y)\in A$.

We now assume $j=j'=0$.
If $\ell< -i$ or $0\leq \ell'$, then $(x,y)$ must be in $A\cup B\cup C$.
So, we assume $-i\leq \ell\leq \ell' <0$.
Consider $\alpha=((m,\ell+i+1),(m',\ell'+i+1))\in \C_{\iinftyclosed}$.
If $\alpha\notin T$ then $(x,y)\cup F_{\iinftyclosed}^\pi(T)$ is not $\NN_\pi$-compatible, a contradiction.
Thus, $\alpha\in T$ and $(x,y)=\Phi(\alpha)\in F_{\iinftyclosed}^\pi(T)$.
Therefore, $F_{\iinftyclosed}^\pi(T)$ is an $\NN_\pi$-cluster.

Let $T\to T'$ be an $\NN_{\iinftyclosed}$-mutation where we exchange $((m,\ell),(m',\ell'))$ and $((n,p),(n',p'))$.
By \cite[Theorem~6.13]{PY21}, there is a quadrilateral in $(\A_{\iinftyclosed},\cfrak_{\iinftyclosed})$ with sides $((m,\ell),(n',p'))$, $((m,\ell),(n,p))$, $((n,p),(m',\ell'))$, and $((m',\ell'),(n',p'))$.
The image of this quadrilateral under $\Phi$ is also a quadrilateral.
Then the image of the diagonals under $\Phi$ are also diagonals.
By \cite[Proposition~6.2.1]{IT15a}, we have an $\NN_\pi$-mutation.

Define $(F_{j,\infty}^\pi,\eta_{j,\infty}^\pi):=(F_{\jinftyclosed}^\pi,\eta_{\jinftyclosed}^\pi)\circ (F_{j,\infty}^{\jinftyclosed}, \eta_{j,\infty}^{\jinftyclosed})$ and $(F_{\jinftyclosed}^\R,\eta_{\jinftyclosed}^\R):= (F_\pi^\R,\eta_\pi^\R)\circ (F_{\jinftyclosed}^\pi, \eta_{\jinftyclosed}^\pi)$, which completes the proof of Theorem~\ref{thm:intro:theories}.

\section{Proof of Theorem~\ref{thm:intro:structures}}\label{sec:strutures in theories}

We now prove the diagram in Figure~\ref{fig:structure diagram} exists and is commutative (Theorem~\ref{thm:intro:structures}).

We begin with a non-example of a restriction of an embedding of cluster theories to an embedding of cluster structures (Example~\ref{xmp:Ainftyclosed not-structure}).
In Section~\ref{sec:restrictions} we show that the embeddings of cluster theories from Sections \ref{sec:embeddings:finite}, \ref{sec:embeddings:finite-infinite}, and \ref{sec:embeddings:isomorphism} restrict to embeddings of cluster structures.
In Sections~\ref{sec:diagram of structures} and \ref{sec:new cont embedding} we define new embeddings of cluster structures.

\begin{example}\label{xmp:Ainftyclosed not-structure}
The embedding of cluster theories $(\Finftyinftyclosed,\etainftyinftyclosed):=(F_{\overline{1,\infty}}^{\overline{\infty}},\eta_{\overline{1,\infty}}^{\overline{\infty}})\circ(F_{1,\infty}^{\overline{1,\infty}},\eta_{1,\infty}^{\overline{1,\infty}})$ (Sections \ref{sec:embeddings:isomorphism}, \ref{sec:emeddings:complete i-infinity-gon}, and \ref{sec:embeddings:2-infinity-gon}) does not restrict to an embedding of cluster structures if we use $\struc{\NN_{1,\infty}}{\C(A_{1,\infty})}$ in Example~\ref{xmp:Ainfty structure}.
In particular, consider the image $\Finftyinftyclosed(T)$ of an $\NN_\infty$-cluster $T$.
If $\Finftyinftyclosed(T)$ has an arc to either $-\infty$ or $+\infty$, that arc is not $\Ninftyclosed$-mutable.
Even if $T$ belongs to $\struc{\NN_\infty}{\C(A_\infty)}$ it is possible that $F_\infty^{\overline{\infty}}(T)$ has an arc to $-\infty$ or $+\infty$ \cite[Theorem~4.4]{HJ12}.
A cluster structure in $\Cinftyclosed$ must be smaller than or notably different than the image of $\struc{\NN_\infty}{\C(A_\infty)}$ under $(\Finftyinftyclosed,\etainftyinftyclosed)$.
\end{example}

\begin{figure}
	\begin{displaymath}
	\xymatrix@C=15ex{
		\NN_m \ar[r]^-{F_m^{1,\infty}(\ref{sec:embeddings:finite-infinite}
,\,\ref{sec:restrictions})} \ar[d]_-{F_m^n(\ref{sec:embeddings:finite},\,\ref{sec:restrictions})} &
		\NN_\infty \ar[r]^-{H_\infty^{1,\infty}(\ref{sec:embeddings:isomorphism}\,\ref{sec:restrictions})} \ar[dr]|-{H_\infty^{i,\infty}(\ref{sec:diagram of structures})} \ar[d]^>{H_{1,\infty}^\pi(\ref{sec:new cont embedding})} &
		\NN_{1,\infty} \ar[d]|-{F_{1,\infty}^{i,\infty}(\ref{sec:diagram of structures})}  \ar[r]^-{F_{1,\infty}^{j,\infty}(\ref{sec:diagram of structures})}  &
		\NN_{j,\infty}
		\\
		\NN_n \ar[r]_-{F_m^{i,\infty}(\ref{sec:embeddings:finite-infinite}
,\,\ref{sec:restrictions})} \ar[ur]|-{F_n^{1,\infty}(\ref{sec:embeddings:finite-infinite}
,\,\ref{sec:restrictions})} &
		\NN_\pi &
		\NN_{i,\infty} \ar[ur]_-{F_{i,\infty}^{j,\infty}(\ref{sec:diagram of structures})} &
	}
	\end{displaymath}
	\caption{Embeddings of cluster structures of type $A$ for any integers $m,n,i,j$ satisfying $2\leq m<n$ and $2\leq i < j$.
		For each cluster structure $\struc{\PP}{\C}$, we have only written its pairwise compatibility condition $\PP$ in order to save space.
		We have also written each embedding $(F,\eta)$ as $F$ and the section(s) in which it is defined.}\label{fig:structure diagram}
\end{figure}

\subsection{Restrictions}\label{sec:restrictions}

Recall the cluster structures $\struc{\NN_n}{\C(A_n)}=\theory{\NN_n}{\C(A_n)}$ (for $n\geq 2$) and $\struc{\NN_\infty}{\C(A_\infty)}$ from Examples~\ref{xmp:An structure} and \ref{xmp:Ainfty structure}, respectively.
It is immediate that $(F_m^n,\eta_m^n)$ is an embedding of cluster structures, for $2\leq m< n$.

To continue, we need the following definition from \cite{HJ12}.
\begin{definition}
	Let $T$ be an $\NN_\infty$-cluster.
	We say $T$ has a \ul{left-fountain (right-fountain) at $k$} if there exists $k\in\Z$ and infinitely-many arcs of the form $(i,k)$ for $i\leq k-2$ (for $k+2\leq i$).
	If there exists a left- and right-fountain at $k$, we say there is a \ul{fountain at $k$}.
\end{definition}

In \cite{HJ12}, the authors show that an $\NN_\infty$-cluster $T$ has a left-fountain if and only if it has a right-fountain.
Futhermore, an $\NN_\infty$-cluster may not have more than one left-fountain or more than one right-fountain.

For any $\NN_n$-cluster $T$, the $\NN_\infty$-cluster $F_n^\infty(T)$ has no left- or right-fountains.
Let $\struc{\NN_\infty}{\C(A_\infty)}$ be full subcategory of $\theory{\NN_\infty}{\C(A_\infty)}$ consisting of those $\NN_\infty$-clusters that do have have a left- or right-fountain.
We immediately see $\struc{\NN_\infty}{\C(A_\infty)}$ is an abstract cluster structure.
Thus, the image of $(F_n^\infty,\eta_n^\infty)$ is inside the cluster structure $\struc{\NN_\infty}{\C(A_\infty)}$.
So, $(F_n^\infty,\eta_n^\infty)$ restricts to an embedding of cluster theories.

Note that $\struc{\NN_\infty}{\C(A_\infty)}$ is \emph{not} the same cluster structure defined by Holm and J{\o}rgensen in \cite{HJ12} (Example~\ref{xmp:Ainfty structure}).
Holm and J{\o}rgensen allow for $\NN_\infty$-clusters that have a fountain.
The isomorphism of cluster theories (from Section~\ref{sec:embeddings:isomorphism}) $\theory{\NN_\infty}{\C(A_\infty)}\to \theory{\NN_{1,\infty}}{\C(A_{1,\infty})}$ restricts to an embedding of cluster structures $\struc{\NN_\infty}{\C(A_\infty)}\to \struc{\NN_{1,\infty}}{\C(A_{1,\infty})}$.

\subsection{New infinite embeddings}\label{sec:diagram of structures}
We define a new embedding of cluster theories $(H_{i,\infty}^{j,\infty},\zeta_{i,\infty}^{j,\infty})$ from $\theory{\NN_{i,\infty}}{\C(A_{i,\infty})}$ to $\theory{\NN_{j,\infty}}{\C(A_{j,\infty})}$, for $0< i < j$,
using an equivalent 2D geometric model of $\theory{\NN_{i,\infty}}{\C(A_{i,\infty})}$.

\begin{definition}
	Let $i>0$ and $Z = \Z$ with the following total order.
	For $m,n\in Z$, if $0\leq_{\Z} m,n$ or $m,n<_{\Z}0$, then we take the same total order in $Z$.
	If $0\leq_{\Z} m$ and $n<_{\Z} 0$, then we say $m<_Zn$.
	Let $\E'_{i,\infty} = Z\times\{1,\ldots,i\}$ with total order $(m,\ell)\leq (n,p)$ if $\ell<_\Z p$ or $\ell=p$ and $m\leq_Z n$.
	Note $\min \E'_{i,\infty}=(0,1)$ and $\max\E'_{i,\infty}=(-1,i)$.
	
	Let $\A'_{i,\infty} = \{(j,\ell)\in\E'_{i,\infty}\times\E'_{i,\infty} \mid \exists k\in\E'_{i,\infty} \text{ s.t. } j<k<\ell\}\setminus\{((0,1),(-1,i))\}$ and
	$\cfrak'_{i,\infty}:\A'_{i,\infty}\times \A'_{i,\infty}\to \{0,1\}$ be the type $A$ crossing function on $\A'_{i,\infty}$.
	Let $\C'(A_{i,\infty})$ be the additive categorification of $(\A'_{i,\infty},\cfrak'_{i,\infty})$ and $\NN'_{i,\infty}$ be the pairwise encoding of $\cfrak'_{i,\infty}$.
\end{definition}
It is straightfoward to check that $\theory{\NN_{i,\infty}}{\C(A_{i,\infty})}\cong\theory{\NN'_{i,\infty}}{\C'(A_{i,\infty})}$.
In \cite[Definition 2.4.6]{IT15b}, the authors provide a limit condition that is automatically met by the $\NN_{1,\infty}$-clusters considered by Holm and J{\o}rgensen in \cite{HJ12} (Example~\ref{xmp:Ainfty structure}).
We take $\struc{\NN'_{i,\infty}}{\C'(A_{i,\infty})}$ to be the restriction of $\theory{\NN'_{i,\infty}}{\C'(A_{i,\infty})}$ to the clusters considered in \cite{IT15b} and, when $i=1$, \cite{HJ12}.
We say an $\NN_{i,\infty}$-cluster $T$ is \ul{locally finite} if, for each $j\in\E'_{i,\infty}$, there exist finitely-many arcs in $T$ with endpoint $j$ (from \cite{IT15b,HJ12}).
Define the locally finite $\NN'_{i,\infty}$-compatible set
\begin{displaymath}
	D := \{ ((-n,1),(n,1)),\ ((-n-1,1),(-n,1)) \mid n<_\Z 0\}\cup\{((0,2),(-1,i))\}.
\end{displaymath}
\begin{definition}
	Let $0< i$ and $j=i+1$.
	Let $\phi:\E'_{i,\infty}\to\E'_{j,\infty}$ be given by $\phi((m,k)) = (m,k+1)$.
	By Lemma \ref{lem:general technique 2} this induces a compatibility preserving injection $\Phi:A'_{i,\infty}\to \A'_{j,\infty}$.
	Let $T$ be an $\NN'_{i,\infty}$-cluster and define $H_{i,\infty}^{j,\infty}(T)$ by $H_{i,\infty}^{j,\infty}(T) := \{\Phi(\alpha)|\alpha\in T\}\cup D$.	
\end{definition}
Since $D$ is locally finite, if an $\NN_{i,\infty}$-cluster $T$ satisfies Igusa and Todorov's limit condition \cite{IT15b}, so will $H_{i,\infty}^{j,\infty}(T)$.
Thus we have an embedding of cluster structures $(H_{i,\infty}^{j,\infty},\zeta_{i,\infty}^{j,\infty})$.
For $j>i+1$ we define 
\begin{displaymath}
	(H_{i,\infty}^{j,\infty},\zeta_{i,\infty}^{j,\infty}) =
	(H_{j-1,\infty}^{j,\infty},\zeta_{j-1,\infty}^{j,\infty})
	\circ\cdots\circ
	(H_{i+1,\infty}^{i+2,\infty},\zeta_{i+1,\infty}^{i+2,\infty})
	\circ
	(H_{i,\infty}^{i+1,\infty},\zeta_{i,\infty}^{i+1,\infty}).
\end{displaymath}
\subsection{New infinte-continuous embedding}\label{sec:new cont embedding}
We now define a new embedding of cluster structures from $\struc{\NN_\infty}{\C(A_\infty)}$ to $\struc{\NN_\pi}{\C_\pi}$.

In \cite{IT15a} the authors show that any indecomposable $X$ in a discrete $\NN_\pi$ cluster $T$ (in the sense of the points in the fundamental domain) without accumulation is $\NN_\pi$-mutable.
In \cite{IT19} the authors prove that $\C_\pi / \add T$, for a discrete $\NN_\pi$-cluster $T$, is abelian.
It is also noted in \cite{IT15a} that if an $\NN_\pi$-cluster $T$ is not discrete then there exists $X\in T$ that is not $\NN_\pi$-mutable.
Thus the cluster structure described in \cite{IT15a,IT19} is a cluster structure as in Definition~\ref{def:cluster structure}.

Thus, we only need to ensure that our embedding takes $\NN_\infty$-clusters without a left- or right-fountain to discrete $\NN_\pi$-clusters.
To do this, we define $b_{i,j,k}$'s similar to $a_{i,j,k}^\ell$'s from Section~\ref{sec:embeddings:continuous}.
For each $i,j,k\in\Z$ such that $0\leq k$ and $0\leq j \leq 2^k$, define $b_{i,j,k} := i + \frac{j}{2^k}$ and
\begin{displaymath}
E:=\left\{(\tan^{-1}b_{i,j,k},\  \tan^{-1}b_{i,j+1,k}) \mid i,j,k\in\Z, \ 0\leq k,\ 0\leq j < 2^k\right\}.
\end{displaymath}

\begin{definition}\label{def:new embedding}
	Let $T$ be an $\NN_\infty$-cluster without a left- or right-fountain.
	Define $\phi:\E_\infty\to\E_\pi$ by $\phi(i) = \tan^{-1}i$.
	By Lemma~\ref{lem:general technique 2} we have $\Phi:\A_\infty\to\A_\pi$ that respects compatibility.
	Now define $H_\infty^\pi (T) := \{\Phi(\alpha) \mid \alpha\in T\}\cup E$.
\end{definition}

In \cite{IT15b} the authors note that any $\NN_{i,\infty}$-cluster which is locally finite

Note that if $T$ is an $\NN_\infty$-cluster with no left- or right-fountain, then $H_\infty^\pi(T)$ is a discrete $\NN_\pi$-cluster.
Using techniques similar to those in Section~\ref{sec:embeddings:continuous} we see that $H_\infty^\pi(T)$ is a functor $\struc{\NN_\infty}{\C(A_\infty)}\to\struc{\NN_\pi}{\C_\pi}$ and we have an embedding of cluster structures $(H_\infty^\pi,\zeta_\infty^\pi)$ where $(\zeta_\infty^\pi)_T((i,j))=\Phi((i,j))$.

If an $\NN_\infty$-cluster $T$ has a fountain then we would have to add a limiting arc in $H_\infty^\pi(T)$ in order to obtain an $\NN_\pi$-cluster.
So, $(H_\infty^\pi,\zeta_\infty^\pi)$ does not extend to an embedding of cluster structures from $\struc{\NN_{1,\infty}}{\C(A_{1,\infty})}$ to $\struc{\NN_\pi}{\C_\pi}$.


\begin{thebibliography}{9}
\bibitem[A21]{A21}
	C.~Amiot,
	\ul{Triangulated categories, equivalences and topological models}.
	Representation Theory [math.RT]. Université Grenoble-Alpes, 2021.
	tel-03279648,
	\href{https://hal.archives-ouvertes.fr/tel-03279648/document}{https://hal.archives-ouvertes.fr/tel-03279648/document}

\bibitem[BG18]{BG18}
	K.~Baur and S.~Graz,
	\emph{Transfinite mutations in the completed infinity-gon},
	Journal of Combinatorial Theory Series A \textbf{155} (2018),
	321--359,
	\href{https://doi.org/10.1016/j.jcta.2017.11.011}{https://doi.org/10.1016/j.jcta.2017.11.011}
	
\bibitem[BIRS09]{BIRS09}
	A.~B.~Buan, O.~Iyama, I.~Reiten, and J.~Scott,
	\emph{Cluster structures for 2-Calabi--Yau categories and unipotent groups},
	Compositio Mathematica \textbf{45} (2009),
	no.~4, 1035--1079 (2009),
	\href{https://doi.org/10.1112/S0010437X09003960}{https://doi.org/10.1112/S0010437X09003960}

\bibitem[BMR+06]{BMRRT06}
	A.~Buan, R.~Marsh, M.~Reineke, I.~Reiten, and G.~Todorov,
	\emph{Tilting theory and cluster combinatorics}, 
	Advances in Mathematics \textbf{204} (2006),
	no.~2, 572--618,
	\href{https://doi.org/10.1016/j.aim.2005.06.003}{https://doi.org/10.1016/j.aim.2005.06.003}

\bibitem[CCS06]{CCS06}
	P.~Caldero, F.~Chapoton, and R.~Schiffler,
	\emph{Quivers with Relations Arising From Clusters ($A_n$ Case)},
	Transactions of the American Mathematical Society \textbf{358} (2006),
	no.~ 3, 1347--1364,
	\href{https://doi.org/10.1090/S0002-9947-05-03753-0}{https://doi.org/10.1090/S0002-9947-05-03753-0}
	
\bibitem[CFZ02]{CFZ02}
	F. Chapoton, S. Fomin, A. Zelevinsky,
	\emph{Polytopal realizations of generalized associahedra},
	Canadian Mathematical Bulletin \textbf{45} (2002),
	no.~4, 537–566,
	\href{https://doi.org/10.4153/CMB-2002-054-1}{https://doi.org/10.4153/CMB-2002-054-1}

\bibitem[FST08]{FST08}
	S.~Fomin, M.~Shapiro, and D.~Thurston, \emph{Cluster algebras and triangulated surfaces. Part I: Cluster complexes},
	Acta Mathematica \textbf{201} (2008),
	83--146,
	\href{https://doi.org/10.1007/s11511-008-0030-7}{https://doi.org/10.1007/s11511-008-0030-7}

\bibitem[FZ02]{FZ02}
	S.~Fomin and A.~Zelevinksy,
	\emph{Cluter algebras I: Foundations},
	Journal of the American Mathematical Society \textbf{15} (2002),
	no.~2, 497--529,
	\href{https://doi.org/10.1090/S0894-0347-01-00385-X}{https://doi.org/10.1090/S0894-0347-01-00385-X}

\bibitem[HJ12]{HJ12}
	T.~Holm and P.~J{\o}rgensen, 
	\emph{On a cluster category of infinite Dynkin type, and the relation to triangulations of the infinity-gon},
	Mathematische Zeitschrift \textbf{270} (2012),
	no.~1, 277--295,
	\href{https://doi.org/10.1007/s00209-010-0797-z}{https://doi.org/10.1007/s00209-010-0797-z}
	
\bibitem[IRT22a]{IRT22a}
	K.~Igusa, J.~D.~Rock, and G.~Todorov,
	\emph{Continuous Quivers of Type A (I) Foundations},
	To appear in Rendiconti del Circolo Matematico di Palermo Series 2.
	
\bibitem[IRT22b]{IRT22b}
	\underline{\qquad},
	\emph{Continuous Quivers of Type A (III) Embeddings of Cluster Theories},
	To appear in the Nagoya Mathematical Journal.
	
\bibitem[IT15a]{IT15a}
	K.~Iguasa and G.~Todorov,
	\emph{Continuous Cluster Categories I},
	Algebras and Representation Theory \textbf{18} (2015),
	no.~1, 65--101,
	\href{https://doi.org/10.1007/s10468-014-9481-z}{https://doi.org/10.1007/s10468-014-9481-z}

\bibitem[IT15b]{IT15b}
	\underline{\qquad},
	\emph{Cluster Categories Coming from Cyclic Posets},
	Communications in Algebra \textbf{43} (2015), no.~10,
	4367--4402,
	\href{https://doi.org/10.1080/00927872.2014.946138}{https://doi.org/10.1080/00927872.2014.946138}
	
\bibitem[IT19]{IT19}
	\underline{\qquad},
	\emph{Continuous cluster categories II: continuous cluster-tilted categories},
	arXiv:1909.05340v1 [math.RT] (2019),
	\href{https://arxiv.org/pdf/1909.05340v1.pdf}{https://arxiv.org/pdf/1909.05340v1.pdf}
	
\bibitem[KMMR21]{KMMR21}
	M.~C.~Kulkarni, J.~P.~Matherne, K.~Mousavand, and J.~Rock,
	\textit{A Continuous Associahedron},
	arXiv:2108.12927 [math.RT] (2021),
	\href{https://arXiv.org/pdf/2108.12927.pdf}{https://arXiv.org/pdf/2108.12927.pdf}
	
\bibitem[PY21]{PY21}
C.~Paquette and E.~Y{\i}ld{\i}r{\i}m, \emph{Completion of discrete cluster categories of type $\mathbb A$}, Transactions of the London Mathematical Soceity \textbf{8} (2021), no.~1, 35--64, \href{https://doi.org/10.1112/tlm3.12025}{https://doi.org/10.1112/tlm3.12025}

\bibitem[R19]{R19}
	J.~D.~Rock,
	\emph{Continuous Quivers of Type $A$ (II) The Auslander--Reiten Space},
	arXiv:1910.04140v1 [math.RT] (2019),
	\href{https://arXiv.org/pdf/1910.04140.pdf}{https://arXiv.org/pdf/1910.04140.pdf}

\bibitem[R20]{R20}
	\underline{\qquad},
	\emph{Continuous Quivers of Type A (IV) Continuous Mutation and Geometric Models of $\EE$-clusters},
	arXiv:2004.11341v2 [math.RT] (2021),
	\href{https://arxiv.org/pdf/2004.11341v2.pdf}{https://arxiv.org/pdf/2004.11341v2.pdf}

\end{thebibliography}
\end{document}